\documentclass[11pt]{amsart}

\usepackage[utf8x]{inputenc}
\usepackage[english]{babel}

\usepackage{graphicx}
\usepackage{paralist}

\usepackage[margin=2.4cm]{geometry}

\usepackage{amssymb,amsmath}
\usepackage[dvipsnames]{xcolor}
\usepackage{bbm}

\usepackage{tikz}

\usepackage[bookmarks,colorlinks,citecolor=magenta, pdftex]{hyperref}

\newcommand\Defn[1]{\textbf{#1}}

\renewcommand\emptyset{\varnothing}

\newcommand\Z{\mathbb{Z}}
\newcommand\Znn{\Z_{\ge0}}

\newcommand\kk{\mathbf{k}}

\newcommand\x{\mathbf{x}}

\newcommand{\upshad}[2]{\tilde{\partial}^{#2}(#1)}

\DeclareMathOperator{\Tor}{Tor}

\newtheorem{thm}{Theorem}%
\newtheorem{cor}[thm]{Corollary}
\newtheorem{prop}[thm]{Proposition}

\newtheorem{conj}{Conjecture}

\theoremstyle{definition}

\newtheorem{ex}{Example}

\title{On $f$- and $h$-vectors of relative simplicial complexes}

\author{Giulia Codenotti}
\address{Fachbereich Mathematik und Informatik, %
Freie Universit\"at Berlin, Berlin, %
Germany}
\email{codenotti@math.fu-berlin.de}

\author{Lukas Katth\"an}
\address{Institut f\"ur Mathematik, Goethe-Universit\"at Frankfurt, Germany}
\email{katthaen@math.uni-frankfurt.de}

\author{Raman Sanyal}
\address{Institut f\"ur Mathematik, Goethe-Universit\"at Frankfurt, Germany}
\email{sanyal@math.uni-frankfurt.de}


\date{\today}

\begin{document}

\begin{abstract}
    A relative simplicial complex is a collection of sets of the form $\Delta
    \setminus \Gamma$, where $\Gamma \subset \Delta$ are simplicial complexes.
    Relative complexes have played key roles in recent advances in algebraic,
    geometric, and topological combinatorics but, in contrast to simplicial
    complexes, little is known about their general combinatorial structure.
    In this paper, we address a basic question in this direction and give a
    characterization of $f$-vectors of relative (multi)complexes on a ground
    set of fixed size. On the algebraic side, this yields a characterization
    of Hilbert functions of quotients of homogeneous ideals over polynomial
    rings with a fixed number of indeterminates. 
    
    Moreover, we characterize $h$-vectors of fully Cohen--Macaulay relative
    complexes as well as $h$-vectors of Cohen--Macaulay relative complexes
    with minimal faces of given dimensions. The latter resolves a question of
    Bj\"orner.
\end{abstract}

\maketitle

\section{Introduction}\label{sec:intro}%

A \Defn{simplicial complex} $\Delta$ is a collection of subsets of a finite
ground set, say $[n] := \{1,\dots,n\}$, such that $\sigma \in \Delta$ and
$\tau \subseteq \sigma$ implies $\tau \in \Delta$. Simplicial complexes are
fundamental objects in algebraic, geometric, and topological combinatorics; see,
for example,~\cite{Stanley96,crt,bjorner}. A basic combinatorial statistic of $\Delta$ is
the \Defn{face vector} (or \Defn{$\boldsymbol f$-vector}) 
\[
    f(\Delta) = (f_{-1},f_0,\dots,f_{d-1}) \, ,
\]
where $f_k = f_k(\Delta)$ records the number of faces $\sigma \in \Delta$ of
dimension $k$, where $\dim \sigma := |\sigma| - 1$ and $d - 1 = \dim \Delta :=
\max \{ \dim \sigma : \sigma \in \Delta\}$. Notice that we allow $\Delta =
\emptyset$, the \emph{void} complex, which is the only complex with
$f_k(\Delta) = 0$ for all $k \ge -1$.  A \Defn{relative simplicial complex}
$\Psi$ on the ground set $[n]$ is the collection of sets $ \Delta \setminus
\Gamma = \{ \tau \in \Delta :  \tau \not \in \Gamma \}$, where $\Gamma \subset
\Delta \subseteq 2^{[n]}$ are simplicial complexes. In general, the pair of
simplicial complexes $(\Delta,\Gamma)$ is not uniquely determined by $\Psi$,
and we call $\Psi = (\Delta, \Gamma)$ a \Defn{presentation} of $\Psi$.  We set
$\dim \Psi := \max \{ \dim \sigma : \sigma \in \Delta \setminus \Gamma \}$.
Relative complexes were introduced by Stanley~\cite{stanley87} and made
prominent recent appearances in, for example,~\cite{AS16,DGKM16, MN, MNY}. The
$f$-vector of a relative complex is given by 
\[
    f(\Psi) \ := \ f(\Delta) - f(\Gamma) \, ,
\]
where we set $f_k(\Gamma) := 0$ for all $k > \dim \Gamma$. When $\Gamma =
\emptyset$, then $\Psi$ is simply a simplicial complex and we write $\Delta$
instead of $\Psi$. We call $\Psi$ a \Defn{proper} relative complex if $\Gamma \neq
\emptyset$ or, equivalently, if $f_{-1}(\Psi) = 0$. 

In contrast to simplicial complexes, much less is known about the
combinatorics of relative simplicial complexes. The first goal of this paper
is to address the following basic question: 
\begin{center}
    \it Which vectors $f = (0,f_0,\dots,f_{d-1}) \in \Z^{d+1}_{\ge0}$ are
    $f$-vectors of proper relative simplicial complexes?
\end{center}
For simplicial complexes, this question is beautifully answered by the
Kruskal--Katona theorem~\cite{kruskal,katona}. Bj\"orner and Kalai~\cite{BK}
characterized the pairs $(f(\Delta),\beta(\Delta))$ where $\Delta$ is a
simplicial complex and $\beta(\Delta)$ is the sequence of Betti numbers of
$\Delta$ (over a field $\kk$). Duval~\cite{duval} characterized the
pairs
$(f(\Delta),f(\Gamma))$ where $\Delta
\subseteq \Gamma$ but, as stated before, the presentation $\Psi = \Delta
\setminus \Gamma$ is generally not unique. Moreover, the following example
shows that a characterization of $f$-vectors of relative complexes is trivial
without further qualifications.

\begin{ex} \label{ex:all_vectors}%
    If $\Delta = 2^{[k+1]}$ is a $k$-dimensional simplex and $\partial \Delta :=
    \Delta \setminus \{[k+1]\}$ denotes its boundary complex, then
    $f_i(\Delta,\partial\Delta) = 1$ if $i = k$ and is zero otherwise. Hence, by
    observing that relative simplicial complexes are closed under
    disjoint unions, any vector $f = (0, f_0,\dots,f_{d-1}) \in
    \Z_{\ge0}^{d+1}$ can occur as the $f$-vector of a proper relative simplicial
    complex.
\end{ex}

The main difference between $f$-vectors of complexes and relative complexes is
that $f_0(\Psi)$ does not reveal the size of the ground set and the
construction outlined in Example~\ref{ex:all_vectors} produces relative
complexes with given $f$-vectors on large ground sets.  Restricting the size
of the ground set is the key to a meaningful treatment of $f$-vectors of
relative complexes.
Therefore, we are going to characterize the $f$-vectors of relative complexes
$\Psi = \Delta \setminus \Gamma$ with $\Gamma \subset \Delta \subseteq 2^{[n]}$
for fixed $n$.
To state our characterization, we need to recall the
binomial representation of a natural number: For any $r,k \in \Z_{\ge 0}$ with
$k > 0$, there are unique integers $ r_k > r_{k-1} > \cdots > r_1 \ge 0$ such
that
\begin{equation}\label{eqn:binomial}
    r \ = \ 
    \binom{r_{k}}{k} +
    \binom{r_{k-1}}{k-1} +
    \cdots +
    \binom{r_{1}}{1} \, .
\end{equation}
We refer the reader to Greene--Kleitman's excellent article~\cite[Sect.~8]{GK}
for details and combinatorial motivations for this and the following
definition. For the representation given in~\eqref{eqn:binomial} we define
\[
    \partial_k(r) \ := \
    \binom{r_{k}}{k-1} +
    \binom{r_{k-1}}{k-2} +
    \cdots +
    \binom{r_{1}}{0} \, .
\]

The Kruskal-Katona theorem characterizes $f$-vectors of simplicial complexes in terms of these $\partial_k(r)$, see Theorem~\ref{thm:KK}.
We prove the following characterization of $f$-vectors of proper relative
complexes in Section~\ref{sec:f-rel}.

\begin{thm}\label{thm:relKK}%
    Let $f = (0,f_0,\dots,f_{d-1}) \in \Znn^{d+1}$ and $n > 0$ and define two
    sequences $(a_0,\dots,a_{d-1})$ and $(b_0,\dots,b_{d-1})$ by
    $a_{d-1} := f_{d-1}$ and $b_{d-1} :=0$ and continue recursively
    \begin{align*}
        a_{k-1} &\ := \ \max(\partial_{k+1}(a_{k}), 
        f_{k-1} + \partial_{k+1}(b_{k}) )   \\
        b_{k-1} &\ := \  \max(\partial_{k+1}(b_{k}), 
  \partial_{k+1}(a_{k})-f_{k-1}  ) 
    \end{align*}
    for $k \ge 0$.
    Then there is a proper relative simplicial complex $\Psi$ on the ground
    set $[n]$ with $f = f(\Psi)$ if and only if $a_0 \le n$.
\end{thm}
The two sequences $(1,a_0,\dots,a_{d-1})$ and $(1,b_0,\dots,b_{d-1})$ are the
componentwise-minimal $f$-vectors of simplicial complexes $\Delta$ and
$\Gamma$ such that $\Gamma \subseteq \Delta$ and $f_{k-1} = f_{k-1}(\Delta) -
f_{k-1}(\Gamma)$ for all $0 \le k < d$.

(Relative) simplicial complexes can be generalized to (relative)
\emph{multicomplexes} by replacing sets with multisets. The notion of an
$f$-vector of a multicomplex is immediate (by taking into account
multiplicities) and the question above carries over to relative multicomplexes
on a ground set of fixed size.  Multicomplexes are more natural from an
algebraic perspective:  If $S := \kk[x_1,\dots,x_n]$ is the polynomial ring
over a field $\kk$ with $n$ indeterminates and $I \subseteq S$ is a monomial
ideal, then the monomials outside $I$ form a (possibly infinite) multicomplex
on ground set $[n]$ and every multicomplex over $[n]$ arises this way. In
particular, the $f$-vector of a multicomplex is the Hilbert function of $S/I$.
By appealing to initial ideals it is easy to see that $f$-vectors of
(infinite) multicomplexes are exactly the Hilbert functions of standard graded
algebras, which were characterized by Macaulay~\cite{macaulay}.  In
Section~\ref{sec:f-rel-mult} we give precise definitions and
Theorem~\ref{thm:relKKm} is the corresponding analogue of
Theorem~\ref{thm:relKK} for proper, possibly infinite, relative
multicomplexes.  The corresponding algebraic statement characterizes Hilbert
functions of $I/J$ where $J \subset I \subseteq S$ are pairs of homogeneous
ideals; see Corollary~\ref{cor:relKKm}.

The \Defn{$\boldsymbol h$-vector} $h(\Psi) = (h_0,\dots,h_d)$ of a
$(d-1)$-dimensional relative complex $\Psi$ is defined through
\begin{equation}\label{eqn:h-vec}
    \sum_{k=0}^d f_{k-1}(\Psi) t^{d-k} \ = \ \sum_{i=0}^d h_{i}(\Psi)
    (t+1)^{d-i} \, .
\end{equation}
Note that if $\dim \Delta = \dim \Gamma$, then $h(\Psi) = h(\Delta) -
h(\Gamma)$.
The $h$-vector clearly carries the same information as the $f$-vector but it
has been amply demonstrated that $h$-vectors often times reveal more
structure; see~\cite{Stanley96} for example.  In particular, if $\Delta$ is a
\Defn{Cohen--Macaulay} simplicial complex (or CM complex, for short) over some
field $\kk$, then $h_i(\Delta) \ge 0$ for all $i \ge 0$.
Stanley~\cite{StanleyCM} showed that Macaulay's theorem characterizing Hilbert
functions of standard graded algebras yields a characterization of $h$-vectors
of CM complexes akin to the Kruskal--Katona theorem. Stronger even, Bj\"orner,
Frankl, and Stanley~\cite{BFS} showed that all admissible $h$-vectors can be
realized by shellable simplicial complexes, a proper subset of CM complexes.

In Section~\ref{sec:macaulay}, we recall the definition of a Cohen--Macaulay
relative complex and we give a characterization of $h$-vectors of
\emph{fully} CM relative complexes. We call a relative complex $\Psi$
\Defn{fully Cohen--Macaulay} over a ground set $[n]$ if it has a presentation
$\Psi = (\Delta,\Gamma)$ with $\Gamma \subset \Delta \subseteq 2^{[n]}$, 
$\dim \Gamma = \dim \Psi$, and $\Psi$ as well as $\Delta$ and $\Gamma$ are
Cohen--Macaulay.

For $r,k \in \Znn$ with $k > 0$, let $r_k > \dots >
r_1 \ge 0$ as defined by~\eqref{eqn:binomial}. We define
\newcommand\Partial{\widetilde{\partial}}%
\[
    \Partial_k(r) \ := \
    \binom{r_{k}-1}{k-1} +
    \binom{r_{k-1}-1}{k-2} +
    \cdots +
    \binom{r_{1}-1}{0} \, .
\]
Note that $\Psi$ is proper if and only if $h_0(\Psi) = 0$. Our
characterization of $h$-vectors of fully CM complexes parallels that of CM
complexes in that it suffices to consider \emph{fully shellable} relative
complexes; see Section~\ref{sec:macaulay} for a definition.

\begin{thm}\label{thm:relM}
    Let $h = (0,h_1,\dots,h_{d}) \in \Znn^{d+1}$ and $n > 0$. Then the
    following are equivalent:
    \begin{enumerate}[\rm (a)]
        \item There is a fully CM relative complex $\Psi$ on ground set $[n]$
            with $h = h(\Psi)$;
        \item There is a fully shellable relative complex $\Psi$ on ground set
            $[n]$ with $h = h(\Psi)$;
        \item Let $(a_0,\dots,a_{d-1})$ and
            $(b_0,\dots,b_{d-1})$ be the sequences defined through $a_{d-1} := h_{d}$ and
            $b_{d-1} :=0$ and recursively continued
    \begin{align*}
        a_{i-1} &\ := \ \max(\Partial_{i+1}(a_{i}), 
        h_{i} + \Partial_{i+1}(b_{i}) )   \\
        b_{i-1} &\ := \  \max(\Partial_{i+1}(b_{i}), 
  \Partial_{i+1}(a_{i})-h_{i}  ) 
    \end{align*}
    for $i \ge 1$. Then $a_0 \le n-d$.     
    \end{enumerate}
\end{thm}

In Section~\ref{sec:fully}, we discuss the difference between CM and fully CM
relative complexes. In particular, we show in Theorem~\ref{thm:nice} that
every $(d-1)$-dimensional CM relative complex has a presentation as a fully CM
relative complex if we allow the ground set to grow by at most $d$ elements.
From this, we derive the following necessary condition on $h$-vectors of
proper CM relative complexes.

\begin{cor}\label{cor:necessary}
    Let $h = (0,h_1,\dots,h_{d}) \in \Znn^{d+1}$ and $n > 0$.  Further, let
    $(a_0,\dots,a_{d-1})$ and $(b_0,\dots,b_{d-1})$ be the sequences defined
    in Theorem~\ref{thm:relM}(c).  If there exists a CM relative complex $\Psi$
    on ground set $[n]$ with $h = h(\Psi)$, then $a_0 \leq n$.
\end{cor}

We conjecture that it actually suffices to extend the ground set by a single
new vertex. This would strengthen the bound of Corollary~\ref{cor:necessary}
to $n-d+1$.

Finally, Theorem~\ref{thm:bjorner} gives a characterization of $h$-vectors of
relative multicomplexes if the dimensions of the minimal faces of $\Psi =
\Delta \setminus \Gamma$ are given. This resolves a question of A.~Bj\"orner
stated in~\cite{stanley87}. 

\bigskip

\textbf{Acknowledgments.} Research that led to this paper was supported by the
National Science Foundation under Grant No.~DMS-1440140 while the authors were
in residence at the Mathematical Sciences Research Institute in Berkeley,
California, during the Fall 2017 semester on \emph{Geometric and Topological
Combinatorics}.  G.C.~was also supported by the Center for International
Cooperation at Freie Universit\"at Berlin and the Einstein Foundation Berlin.
L.K.~was also supported by the DFG, grant KA 4128/2-1.  R.S.~was also
supported by the DFG Collaborative Research Center SFB/TR 109 ``Discretization
in Geometry and Dynamics''. We thank the two referees for helpful suggestions.

\section{\texorpdfstring{$f$}{f}-vectors of relative simplicial complexes}
\label{sec:f-rel}
\newcommand\f{\mathbf{f}}%
\newcommand\relBnd{\partial^\mathsf{rel}}%

\newcommand\F{\mathcal{F}}%
\newcommand\C[1]{\mathrm{C}{#1}}%
The proof of Theorem~\ref{thm:relKK} follows the same ideas as that of the classical Kruskal--Katona theorem given in~\cite[Sect.~8]{GK}.
A simplicial complex $\Delta \subset 2^{[n]}$ is called \Defn{compressed} if its set of $k$-faces forms an initial segment with respect to the reverse lexicographic order on the $(k+1)$-subsets of $[n]$, for each $k$.
Note that if $\Delta$ and $\Gamma$ are both compressed simplicial complexes and $f_k(\Gamma) \leq f_k(\Delta)$ for all $k$, then $\Gamma \subseteq \Delta$.
The Kruskal--Katona theorem now states that $f$ is the $f$-vector of a simplicial complex if and only if it is the $f$-vector of a compressed simplicial complex, which can be checked by numerical conditions.
\begin{thm}[{Kruskal~\cite{kruskal}, Katona~\cite{katona}}]\label{thm:KK}
	For a vector $f = (1,f_0,\dots,f_{d-1}) \in \Znn^{d+1}$, the following conditions are equivalent:
	\begin{enumerate}[\rm (a)]
		\item $f$ is $f$-vector of a simplicial complex;
		\item $f$ is $f$-vector of a compressed simplicial complex;
		\item $\partial_{k+1}(f_{k}) \le f_{k-1}$ for all $k \ge 1$.
	\end{enumerate}
\end{thm}

The shadow of a family of $k$-sets consists of all
$(k-1)$-subsets of the $k$-sets of the family. The
Kruskal-Katona theorem tells us that $\partial_{k+1}(r)$ is
the minimum size of the shadow of a family $k$-sets of size
$r$. Actually, this minimum is always achieved if the family
is compressed. Note that this implies in particular that the
functions $\partial_k$ are monotone.

With these preparations, we can now give the proof of our Theorem~\ref{thm:relKK}.

\begin{proof}[Proof of Theorem~\ref{thm:relKK}]
Let us recall the definition of the sequences $(a_0,\dots,a_{d-1})$ and $(b_0,\dots,b_{d-1})$.
We have that
$a_{d-1} = f_{d-1}$, $b_{d-1} =0$ and 
\begin{align*}
    a_{k-1} &\ = \ \max(\partial_{k+1}(a_{k}), f_{k-1} + \partial_{k+1}(b_{k}) ) &=&\ \partial_{k+1}(a_{k}) + \max(0, f_{k-1} - (\partial_{k+1}(a_{k})-  \partial_{k+1}(b_{k})) );  \\
    b_{k-1} &\ = \  \max(\partial_{k+1}(b_{k}), \partial_{k+1}(a_{k})-f_{k-1}
    ) &=&\ \partial_{k+1}(b_{k}) + \max(0 ,  (\partial_{k+1}(a_{k})-
    \partial_{k+1}(b_{k})) - f_{k-1}  ) ,
\end{align*}
for $1 \leq k \leq d-1$.
From the second expression for $a_{k-1}$ and $b_{k-1}$ it is easy to see that $a_{k-1} - b_{k-1} = f_{k-1}$.
In particular, we have that $a_k \geq b_k$ for $k \geq 0$.

We now show the sufficiency of the condition, so assume that $a_0 \leq n$.  As
both sequences $(1,a_0,\dots,a_{d-1})$ and $(1,b_0,\dots,b_{d-1})$ satisfy the
condition of the Kruskal-Katona theorem (Theorem~\ref{thm:KK}), there exist
compressed simplicial complexes $\Gamma, \Delta \subset 2^{[n]}$ whose
respective $f$-vectors equal the two sequences.  In particular, since both
complexes are compressed and $f_k(\Gamma) = b_k \leq a_k = f_k(\Delta)$, it
holds that $\Gamma \subset \Delta$, and the relative complex
$\Psi:=(\Delta,\Gamma)$ has $f$-vector $f$.

Now we turn to the necessity of our condition.  Assume that we are given a
relative complex $\Psi = (\Delta, \Gamma)$ on the ground set $[n]$ with
$f(\Psi) = f$.  We show by induction on $k$ that $a_k \leq f_k(\Delta)$ and
$b_k \leq f_k(\Gamma)$ for $k \geq 0$.

The base case $k = d-1$ is obvious.  For the inductive step, it follows from
Theorem~\ref{thm:KK} that $f_{k-1}(\Delta) \geq \partial_{k+1}(f_{k}(\Delta))$,
and further $f_{k}(\Delta) \geq a_k$ implies that $\partial_{k+1}(f_{k}(\Delta))
\geq \partial_{k+1}(a_k)$.  Similarly, it holds that $f_{k-1}(\Delta) = f_{k-1}
+ f_{k-1}(\Gamma) \geq f_{k-1} + \partial_{k+1}(f_{k}(\Gamma)) \geq f_{k-1} +
\partial_{k+1}(b_k)$. Together, this implies that 
\[ 
    f_{k-1}(\Delta) \ \geq \ \max(\partial_{k+1}(a_k), f_{k-1} +
    \partial_{k+1}(b_k)) \ = \ a_{k-1} \, .
\]
Moreover, the last inequality together with the fact that $f_{k-1}(\Delta) -
f_{k-1}(\Gamma) = a_{k-1} - b_{k-1}$ implies that $f_{k-1}(\Gamma) \geq
b_{k-1}$.  In particular, $a_0 \leq f_0(\Delta) \leq n$.
\end{proof}

\section{\texorpdfstring{$f$}{f}-vectors of relative multicomplexes}
\label{sec:f-rel-mult}%
\newcommand\tDelta{\widetilde{\Delta}}%
\newcommand\tGamma{\widetilde{\Gamma}}%
\newcommand\tPsi{\widetilde{\Psi}}%
\newcommand\tF{\widetilde{\F}}%
A \Defn{$\boldsymbol k$-multiset} is a set with repetitions
allowed. 
A \Defn{multicomplex} $\tDelta$ is a collection of multisets closed under taking (multi-)subsets.
We denote a $k$-multisubset of $[n]$ by $F = \{s_1, s_2, \dots,s_k\}_\le$ where $1 \le s_1 \le s_2 \le\cdots \le s_k \le n$.
We say that the dimension of $F$ is $k-1$ and in the same way as for simplicial complexes, one defines $f$-vectors of multicomplexes.
Note that multicomplexes can be infinite, even if the ground set is finite.

The sequences which arise as $f$-vectors of multicomplexes are called
\Defn{$\boldsymbol M$-sequences} and they have a well-known classification due
to Macaulay.  Namely, a sequence $(1, f_0, f_1, \dots)$ is an $M$-sequence if
and only if $f_{k-1} \geq \Partial_{k+1}(f_{k})$.  Moreover, as in the
simplicial case, for each $M$-sequence $f$ there exists a unique
\emph{compressed} multicomplex $\tDelta$ with $f = f(\tDelta)$. Here, being
compressed is defined as in the simplicial case.
We refer the reader to \cite[Sect.~8]{GK} of \cite[Sect.~II.2]{Stanley96} for
details.  

Using compressed multicomplexes and the characterization of $M$-sequences, the
same proof as for Theorem \ref{thm:relKK} also yields the following
characterization for $f$-vectors of finite proper relative multicomplexes
$\tPsi = (\tDelta,\tGamma)$.

\begin{thm}\label{thm:relKKm}
    Let $f = (0,f_{0},\dots,f_{d-1}) \in \Znn^{d+1}$ and $n > 0$ and define two
    sequences $(a_0,\dots,a_{d-1})$ and $(b_0,\dots,b_{d-1})$ by
    $a_{d-1} := f_{d-1}$ and $b_{d-1} :=0$ and continue recursively
    \begin{align*}
        a_{k-1} &\ := \ \max(\Partial_{k+1}(a_{k}), 
        f_{k-1} + \Partial_{k+1}(b_{k}) )   \\
        b_{k-1} &\ := \  \max(\Partial_{k+1}(b_{k}), 
  \Partial_{k+1}(a_{k})-f_{k-1}  ) 
    \end{align*}
    for $k \ge 0$.
    Then there is a proper (finite) relative multicomplex $\tPsi$ on the ground set $[n]$
    with $f = f(\tPsi)$ if and only if $a_0 \le n$.
\end{thm}

Now we turn to the classification of $f$-vectors of not necessarily finite multicomplexes.
In the proof of Theorem \ref{thm:relKK}, it was crucial that relative simplicial complexes have bounded dimension, so that we could proceed by induction from the top dimension downwards.
For general relative multicomplexes, we will instead proceed from dimension $0$ upwards.
This requires some new notation.
For $r,k \in \Znn$ with $k > 0$, let $r_k > \dots >
r_1 \ge 0$ as defined by~\eqref{eqn:binomial}. We define
\[
\upshad{r}{k} \ := \
\binom{r_{k}+1}{k+1} +
\binom{r_{k-1}+1}{k+2} +
\cdots +
\binom{r_{1}+1}{2} \, .
\]
It is not difficult to see that $\Partial_{k+1}(\upshad{r}{k}) = r$ and $\upshad{\Partial_{k}(r)}{k-1} \geq r$.
Therefore, $M$-sequences can be equivalently characterized as those sequences $(f_{-1}, f_0, \dotsc)$ which satisfy $f_{k+1} \geq \upshad{f_k}{k+1}$ for all $k$.

\begin{thm}\label{thm:relKKmi}
    Let $f = (0,f_{0},f_1,\dots)$ be a sequence of non-negative integers and
    $n > 0$ and define two sequences $(a_0,a_1, \dots)$ and $(b_0,b_1, \dots)$
    by $a_{0} := n$, $b_{0} := n - f_0$ and continue recursively
	\begin{align*}
	    a_{k+1} &\ := \ \min(\upshad{a_k}{k+1}, f_{k+1} + \upshad{b_k}{k+1})\\
	    b_{k+1} &\ := \ \min(\upshad{b_k}{k+1}, \upshad{a_k}{k+1} - f_{k+1} )
	\end{align*}
    for $k \ge 0$.  Then, there is a proper relative multicomplex $\tPsi$ on
    the ground set $[n]$ with $f = f(\tPsi)$ if and only if $b_k \geq 0$ for
    all $k \geq 0$.
\end{thm}
The proof is almost the same as the proof of Theorem~\ref{thm:relKK}, using the characterization of $M$-sequences in terms of $\tilde{\partial}^{k}$.
The only difference is that to prove necessity, one needs to start
the induction at $k=0$ and proceed in increasing order.

The classical theorem by Macaulay characterizes Hilbert functions of standard
graded algebras, and Theorem \ref{thm:relKKmi} has a similar interpretation.
We denote the Hilbert function of a finitely generated graded module $M$ over the polynomial ring $\kk[x_1,\dotsc,x_n]$ by $H(M, k) := \dim_\kk M_k$.
\begin{cor}[Macaulay for quotients of ideals]\label{cor:relKKm}
    Let $H : \Znn \to \Znn$ with $H(0) = 0$ and $n \geq H(1)$.  Furthermore, let
    $(a_0,a_1, \dots)$ and $(b_0,b_1, \dots)$ be the two sequences of
    Theorem~\ref{thm:relKKmi}, where we set $f_k = H(k+1)$.  Then, there exist
    two proper homogeneous ideals $J \subset I \subsetneq
    \kk[x_1,\dotsc,x_n]$ with $H(k) = H(I/J,k)$ for all $k$, if and only if
    $b_k \geq 0$ for all $k \geq 0$.
\end{cor}
\begin{proof}
	Consider a homogeneous ideal $I \subseteq \kk[x_1,\dotsc,x_n]$.
	For any fixed term order $\preceq$, the collection of standard monomials, that is, the monomials not contained in the initial ideal of $I$ with respect to $\preceq$, is naturally identified with a multicomplex $\tDelta$.
	Since the standard monomials form a vector space basis of $\kk[x_1,\dotsc,x_n]/I$ that respects the grading, the $f$-vector of $\tDelta$ coincides with the Hilbert function of $\kk[x_1,\dotsc,x_n]/I$.
	Moreover, if $J \subseteq I \subseteq \kk[x_1,\dotsc,x_n]$ are two homogeneous ideals, then passing to the initial ideals (with respect to $\preceq$) preserves the inclusion.
	Therefore, any Hilbert function of a quotient of ideals also arises as $f$-vector of a relative multicomplex.
	
	For the converse we associate to any multicomplex $\tDelta$ the monomial ideal corresponding to all multisets not in $\tDelta$.
\end{proof}

\section{\texorpdfstring{$h$}{h}-vectors of relative Cohen-Macaulay complexes}\label{sec:macaulay}

Let $\Psi = (\Delta,\Gamma)$ be a $(d-1)$-dimensional relative simplicial
complex and let $\sigma_1,\dots,\sigma_m$ be some ordering of the
inclusion-maximal faces (i.e., the facets) of $\Psi$.  Define
\[
    \Psi_j \ := \ \left( 2^{\sigma_1} \cup 2^{\sigma_2} \cup \dots \cup
    2^{\sigma_j} \right) \cap (\Delta \setminus \Gamma) 
\]
for $j \ge 1$ and set $\Psi_0 := \emptyset$. We call the ordering of the facets a \Defn{shelling
order} if $\Psi_{j} \setminus \Psi_{j-1}$ has a unique inclusion-minimal
element $R(\sigma_j)$ for all $j=1,\dots,m$. Consequently, $\Psi$ is
\Defn{shellable} if it has a shelling order. If $\Gamma = \emptyset$ and hence
$\Psi$ is a simplicial complex, this recovers the usual notion of
shellability. The $h$-vector $h(\Psi)$ of a shellable relative complex has a
particularly nice interpretation:
\[
    h_i(\Psi) \ = \ |\{ j : |R(\sigma_j)| = i \}| \, ,
\]
for $0 \le i \le d$. It is shown in~\cite[Sect.~III.7]{Stanley96} that a shellable
relative complex is Cohen--Macaulay but the converse does not need to hold.

We will call a relative complex $\Psi$ \Defn{fully shellable}  if it has a
presentation $\Psi = (\Delta,\Gamma)$ such that $\dim \Psi = \dim \Gamma$ and
$\Psi$ as well as $\Delta$ and $\Gamma$ are shellable. By the above remarks,
it is clear that fully shellable relative complexes are fully Cohen--Macaulay
and, again, the converse does not necessarily hold.

In light of Theorem~\ref{thm:relKKm}, condition (c) of
Theorem~\ref{thm:relM} states that $h$ is the $f$-vector of a proper relative
multicomplex. In order to prove the implication (c) $\Longrightarrow$ (b), we
will show that for every relative multicomplex on the ground set $[n-d]$ with given
$f$-vector $h = (0,h_1,\dots,h_d)$, there is a fully shellable relative
complex $\Psi$ with $h(\Psi) = h$. 

Let $\tPsi = (\tDelta,\tGamma)$ be a proper relative
$(d-1)$-dimensional multicomplex on ground set $[n-d]$ and
assume that $\tDelta$ and $\tGamma$ are compressed.  To turn
$\tPsi$ into a relative complex, we follow the construction
in~\cite{BFS}.  Order the collection of multisets of size
$\leq d$ on the ground set $[n-d]$ by graded reverse
lexicographic order, and the collection of $d$-sets on $[n]$
by reverse lexicographic order.  There is a unique bijection
$\Phi_d$ between these two collections which preserves the
given orders. Explicitly, the map is 
\[
    F = \{b_1, b_2, \dots, b_k \}_{\le} \ \mapsto \ \Phi_d(F) \ := \
    \{1,2,\dots,d-k, b_1 + d-k+1, b_2 + d-k+2, \dots, b_k + d \} \, .
\]
We denote by $\Delta$ the simplicial complex with
facets $\{ \Phi_d(F) : F \in \tDelta\}$ and $\Gamma$ likewise. Since
$\tGamma$ is a submulticomplex of $\tDelta$, it follows that $\Gamma \subset
\Delta$ and $\Psi = (\Delta,\Gamma)$ is a relative complex with $\dim \Psi =
\dim \Delta = \dim \Gamma = d - 1$.

\begin{prop}\label{prop:BFSrelShell}

    Let $\tPsi = (\tDelta,\tGamma)$ be a $(d-1)$-dimensional relative
    multicomplex such that $\tDelta$ and $\tGamma$ are compressed. Let $\Psi =
    (\Delta, \Gamma)$ be the corresponding relative simplicial complex
    constructed above. Given  an ordering $\prec$ of the faces of $\tDelta$
    such that $F \prec F'$ whenever $|F| < |F'|$, the induced ordering on the
    facets $\Phi_d(F)$ of $\Delta$ is a shelling order for $\Delta$, $\Gamma$,
    and $\Psi$.
\end{prop}

\begin{proof}
    It was shown in~\cite{BFS} that any such ordering gives a shelling order for
    $\Delta$ with restriction sets
    \[
        R(\sigma) \ = \ \sigma \setminus \{1,2,\dots,d-k\} \ = \ \{s_1 +
        d-k+1,\dots,s_k + d\}
    \]
    if $\sigma = \Phi_d(\{s_1,\dots,s_k\}_\le)$.
    We are left to prove that restricting this order
    to the facets of $\Delta \setminus \Gamma$ yields a shelling order for
    $\Psi$. It suffices to show that if $\sigma$ is a facet of $\Psi$, i.e., a
    facet of $\Delta$ not contained in $\Gamma$, then $R(\sigma) \not\in
    \Gamma$.

    Let $F = \{s_1, \dots, s_k\}_{\leq}$ be the face of $\tDelta$ such that
    $\sigma = \Phi_d(F)$. We will show that any facet $\sigma'$ of $\Delta$
    which contains $R := R(\sigma)$ does not belong to $\Gamma$. By
    construction, the facets of $\Gamma$ are a subset of the facets of $\Delta$,
    and thus  $R \notin \Gamma$.

    Let $\sigma'$ be a facet of $\Delta$ which contains $R$ and let $F'$ be
    the corresponding element of $\tDelta$ with $\sigma' = \Phi_d(F')$.
    Observe that either $\sigma' = \sigma$ or $t = |F'| > |F| = k$. Indeed, if $t<k$,
    $\{1, 2, \dots, d-k+1\} \subseteq \sigma'$, and since $R \cap \{1, 2,
    \dots, d-k+1\} = \emptyset$, $R$ cannot be a subset of $\sigma'$.
    If $t=k$, then $\sigma' \supseteq R$ implies $\sigma' = \sigma$.

    So, let us assume that $t > k$. Let $G = \{r_1, \dots, r_t\}_{\leq}$ be
    the smallest $t$-multiset in $\tDelta$ in reverse lexicographic order such that
    $\tau = \Phi_d(G) \supseteq R$. Now $\tau = \{1, \dots, d-t\} \cup S$,
    with $S= \{d-t +1 +r_1, \dots, d+r_t\}$. As before, observe that $R \cap
    \{1, \dots, d-t\} = \emptyset$. Since $\Phi_d$ preserves the reverse
    lexicographic order on $t$-multisets, $S$ is also minimal with respect to
    reverse lexicographic order. Therefore the elements of $R$ are the largest
    elements in $S$ and 
    \[
        G \ = \ \{\underbrace{1, \dots, 1}_{t-k}, s_1, \dots, s_k\}_{\leq}.
    \]
    Then $F = \{s_1, \dots, s_k\}_{\leq} \subseteq G$, and since $F\notin
    \tGamma$ and $\tGamma$ is a multicomplex, it follows that $G \notin
    \tGamma$. Since $\tGamma$ is compressed and $G$ is smaller than $F'$, $F'$
    also does not belong to $\tGamma$. This implies $\sigma \not \in \Gamma$.
\end{proof}

\begin{proof}[Proof of 
Theorem~\ref{thm:relM}: (c) $\Longrightarrow$ (b) $\Longrightarrow$ (a)]
By Theorem~\ref{thm:relKKm}, condition (c) guarantees the existence of a
proper relative multicomplex $\tPsi$ with $f$-vector $h$. 
By Proposition~\ref{prop:BFSrelShell}, the construction
above yields a fully shellable relative simplicial complex $\Psi$ 
with $h = h(\Psi)$.  This proves (c) $\Longrightarrow$ (b).
Theorem 2.5 for relative complexes in~\cite{Stanley96}
asserts that $\Psi$ is fully Cohen--Macaulay and hence proves 
(b) $\Longrightarrow$ (a).
\end{proof}

In order to prove the implication (a) $\Longrightarrow$ (c), we make use of
the powerful machinery of Stanley--Reisner modules. Let $\kk$ be an infinite
field. For a fixed $n > 0$, let $S := \kk[x_1,\dots,x_n]$ be the polynomial
ring. For a simplicial complex $\Delta \subseteq 2^{[n]}$, its
\Defn{Stanley--Reisner ideal} is $I_\Delta := \langle \x^\tau : \tau \not\in
\Delta \rangle$ and we write $\kk[\Delta] := S/I_\Delta$ for its
\Defn{Stanley--Reisner ring}. If $\Gamma \subset \Delta$ is a pair of
simplicial complexes, then $\kk[\Delta] \twoheadrightarrow \kk[\Gamma]$ and the
\Defn{Stanley--Reisner module} of $\Psi = (\Delta,\Gamma)$ is
\newcommand\SRmod{\mathrm{M}}
\[
    \SRmod[\Psi] \ := \ \ker( \kk[\Delta] \twoheadrightarrow \kk[\Gamma]) \ = \
    I_\Gamma / I_\Delta \, .
\]
This is a graded $S$-module and $\Psi$ is a \Defn{Cohen--Macaulay} relative
complex if $\SRmod[\Psi]$ is a Cohen--Macaulay module over $S$. In
particular, any choice of generic linear forms $\theta_1,\dots,\theta_d \in S$
for \mbox{$d = \dim \Psi + 1$}
is a regular sequence for $\SRmod[\Psi]$ and 
\[
    \dim_\kk ( \SRmod[\Psi] / \langle \theta_1,\dots,\theta_d \rangle \SRmod[\Psi])_i \ = \
    h_i(\Psi) \, ,
\]
for all $i \ge 0$.

\begin{proof}[Proof of Theorem~\ref{thm:relM}: (a) $\Longrightarrow$ (c)]
    Let $(\Delta,\Gamma)$ be a presentation of $\Psi$ such that $\dim \Gamma =
    \dim \Psi$ and $\Delta$ and $\Gamma$ are CM.  Consider the short exact
    sequence 
    \begin{equation}\label{eq:seq1}
		0 \ \to \ \SRmod[\Psi] \ \to \ \kk[\Delta] \ \to \ \kk[\Gamma] \
        \to \ 0
    \end{equation}
    of $S$-modules. Let $\theta \in S$ be a generic linear form.
    Tensoring~\eqref{eq:seq1} with $S / \theta$ yields
    \begin{equation}\label{eq:seq2}
        \Tor_1^S(\kk[\Gamma], S/\theta) \ \to \ \SRmod[\Psi]/\theta\SRmod[\Psi]
        \ \to \ \kk[\Delta]/ \theta\kk[\Delta] \ \to \ \kk[\Gamma]/
        \theta\kk[\Gamma] \ \to \ 0
    \end{equation}
    By resolving $S/\theta$, it is easy to see that $\Tor_1^S(\kk[\Gamma],
    S/\theta) = (0 :_{\kk[\Gamma]} \theta) = 0$, so \eqref{eq:seq2} is a short
    exact sequence as well.

    By our choice of presentation, $\kk[\Gamma]$ is Cohen--Macaulay and
    we may repeat the process for a full regular sequence $\Theta = (\theta_1,
    \dotsc, \theta_{d})$ to arrive at 
    \begin{equation}\label{eq:seq3}
        0 \ \to \ \SRmod[\Psi]/\Theta\SRmod[\Psi] \ \to \ \kk[\Delta]/
        \Theta\kk[\Delta] \ \to \ \kk[\Gamma]/ \Theta\kk[\Gamma] \ \to \ 0 \, .
    \end{equation}

    \newcommand\In{\mathrm{in}_\preceq}%
    Since $\Psi$ is Cohen--Macaulay, the Hilbert function of
    $\SRmod[\Psi]/\Theta\SRmod[\Psi]$ is exactly the $h$-vector of $\Psi$ and,
    moreover, we can identify $\SRmod[\Psi]/\Theta\SRmod[\Psi]$ with a graded
    ideal in $\kk[\Delta]/ \Theta\kk[\Delta]$. By a linear change of coordinates,
    this yields a pair of homogeneous ideals $J_\Delta \subset J_\Gamma
    \subset R := \kk[y_1,\dots,y_{n-d}]$ with difference of Hilbert functions
    exactly $h(\Psi)$. For any fixed term order $\preceq$, we denote by
    $\In(J_\Delta), \In(J_\Gamma)$ the corresponding initial ideals.  The
    passage to initial ideals leaves the Hilbert functions invariant and
    $\In(J_\Delta) \subseteq \In(J_\Gamma)$; c.f.~\cite[Prop.~9.3.9]{CLOS}.
    The corresponding collections of standard monomials are naturally
    identified with a pair of multicomplexes $\tGamma \subset \tDelta$  with
    $f$-vector $h$ and this completes the proof.
\end{proof}

\section{Cohen--Macaulay versus fully Cohen--Macaulay}
\label{sec:fully}%

Theorem~\ref{thm:relM} only addresses the characterization of $h$-vectors of
fully CM relative complexes. By definition, a relative simplicial complex
$\Psi$ is the set difference of a pair $\Gamma \subset \Delta \subseteq
2^{[n]}$ of simplicial complexes.  This presentation is by no means unique and
it is natural to ask if in the case that $\Psi$ is Cohen--Macaulay, there are
always CM complexes $\Gamma' \subseteq \Delta' \subseteq 2^{[n]}$ of dimension
$\dim \Psi$ such that $\Psi = \Delta' \setminus \Gamma'$. The following
example shows that this is not the case.

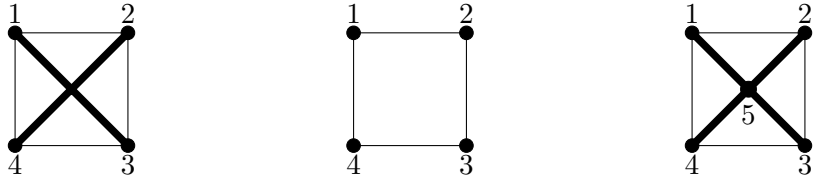
\begin{figure}[t]
	\begin{tikzpicture}[scale=1.5]
	\newcommand{\gammacolor}{gray!40}

	\begin{scope}
		\coordinate (v1) at (0,0);
		\coordinate (v2) at (1,0);
		\coordinate (v3) at (1,-1);
		\coordinate (v4) at (0,-1);
		
		\draw (v1)--(v2)--(v3)--(v4)--cycle;
		\draw[line width=1mm] (v1) -- (v3)  (v2) -- (v4);
		
		\foreach \p in {(v1),(v2),(v3),(v4)}
		\draw[fill=black] \p circle (0.06);
		
		\path (v1) node[anchor=south] {$1$};
		\path (v2) node[anchor=south] {$2$};
		\path (v3) node[anchor=north] {$3$};
		\path (v4) node[anchor=north] {$4$};
	\end{scope}
	
	\begin{scope}[xshift=3cm]
		\coordinate (v1) at (0,0);
		\coordinate (v2) at (1,0);
		\coordinate (v3) at (1,-1);
		\coordinate (v4) at (0,-1);
		
		\draw (v1)--(v2)--(v3)--(v4)--cycle;
		
		\foreach \p in {(v1),(v2),(v3),(v4)}
		\draw[fill=black] \p circle (0.06);
		
		\path (v1) node[anchor=south] {$1$};
		\path (v2) node[anchor=south] {$2$};
		\path (v3) node[anchor=north] {$3$};
		\path (v4) node[anchor=north] {$4$};
	\end{scope}
	
	\begin{scope}[xshift=6cm]
		\coordinate (v1) at (0,0);
		\coordinate (v2) at (1,0);
		\coordinate (v3) at (1,-1);
		\coordinate (v4) at (0,-1);
		\coordinate (v5) at (0.5,-0.5);
		\coordinate (v5p) at (0.5,-0.55);
		
		\draw (v1)--(v2)--(v3)--(v4)--cycle;
		\foreach \p in {(v1),(v2),(v3),(v4)}
			\draw[line width=1mm] (v5) -- \p;
		
            \foreach \p in {(v1),(v2),(v3),(v4)}
			\draw[fill=black] \p circle (0.06);

        \draw[fill=black] (v5) circle (0.07);
		
		\path (v1) node[anchor=south] {$1$};
		\path (v2) node[anchor=south] {$2$};
		\path (v3) node[anchor=north] {$3$};
		\path (v4) node[anchor=north] {$4$};
		\path (v5p) node[anchor=north] {$5$};
	\end{scope}
	\end{tikzpicture}
    \caption{The relative complexes of Example \ref{ex:1}, Example \ref{ex:2},
    and Example \ref{ex:3}. In each case, $\Gamma$ is drawn in 
    bold.}\label{fig:examples}
\end{figure}

\begin{ex}\label{ex:1}
    Let $\Delta \subset 2^{[4]}$ be the complete graph on $4$ vertices, that
    is, the complex consisting of all subsets of $[4]$ of size at most $2$.
    Let $\Gamma \subset \Delta$ be a perfect matching, see Figure \ref{fig:examples}. Then $\Delta \setminus
    \Gamma$ is the relative complex consisting of $4$ \emph{open} edges. This is
    a shellable relative complex. It is easy to check that on the fixed
    ground set $[4]$, this is the only presentation with $\dim \Delta = \dim
    \Gamma = 1$ and hence $\Psi$ is not fully Cohen--Macaulay.
\end{ex}

There are several possibilities to weaken the requirements on fully
Cohen--Macaulay, for example, the requirement that $\dim \Gamma = \dim \Psi$.
The next example, however, shows that the characterization of
Theorem~\ref{thm:relM} then ceases to hold.

\begin{ex}\label{ex:2}
    Let $\Delta \subseteq 2^{[4]}$ be the $1$-dimensional complex with facets
    $\{1,2\}, \{2,3\}, \{3,4\}, \{1,4\}$ and let $\Gamma$ be the complex
    composed of the vertices of $\Delta$. Then $\Psi = (\Delta,\Gamma)$ is a
    relative complex isomorphic to the relative complex of Example~\ref{ex:1}.
    Both $\Delta$ and $\Gamma$ are Cohen--Macaulay but $\dim \Gamma < \dim
    \Psi$. In particular, $\Psi$ is shellable with $h$-vector $h := h(\Psi) =
    (0,0,4)$. However, $h$ is not the $f$-vector of a relative multicomplex on
    ground set $[4-2]$, as any such (relative) multicomplex can have at most
    $3$ faces of dimension $1$.
\end{ex}

Nevertheless, it is possible to remedy the problem illustrated in
Example~\ref{ex:1} by allowing more vertices.

\begin{ex}\label{ex:3}
    Let $\Psi = (\Delta,\Gamma)$ be the relative complex of
    Example~\ref{ex:1}. Let $\Delta' := \Delta \cup \{ \{i,5\} : i \in [4]\}$
    be the graph-theoretic cone over $\Delta$ and define $\Gamma'$
    accordingly. Then $\Delta \setminus \Gamma = \Delta' \setminus \Gamma'$
    and, since $\Delta'$ and $\Gamma'$ are connected graphs and hence
    Cohen--Macaulay, this shows that $\Psi$ is a fully Cohen--Macaulay relative
    complex over the ground set $[5]$.
\end{ex}

The following result now shows that every Cohen--Macaulay relative complex is
fully Cohen--Macaulay if the ground set is sufficiently enlarged.

\begin{thm}\label{thm:nice}
    Let $\Gamma \subset \Delta \subseteq 2^{[n]}$ be simplicial complexes,
    such that $\Psi = (\Delta, \Gamma)$ is Cohen-Macaulay of dimension $d-1$.
    Let $e$ be the depth of $\kk[\Gamma]$.
    Then there exist $\Gamma' \subseteq \Delta' \subseteq 2^{[n+d-e]}$, such
    that $\Delta' \setminus \Gamma' = \Delta \setminus \Gamma$, and both
    $\Delta'$ and $\Gamma'$ are Cohen-Macaulay of dimension $d-1$.
\end{thm}
\begin{proof}
    Let $\Gamma_1$ be the $(d-e)$-fold cone over $\Gamma$ and set $\Delta_1 :=
    \Delta \cup \Gamma_1$. Then $\Delta_1 \setminus \Gamma_1 = \Delta
    \setminus \Gamma$. Further note that $\kk[\Gamma_1] =
    \kk[\Gamma][y_1,\dotsc, y_{d-e}]$, where the $y_i$ are new variables.
    Thus, the depth of $\kk[\Gamma_1]$ is $d$. Finally, we define $\Delta'$ and
    $\Gamma'$ to be the $(d-1)$-dimensional skeleta of $\Delta_1$ and
    $\Gamma_1$, respectively. Again, $\Delta' \setminus \Gamma' = \Delta
    \setminus \Gamma$ and thus $\Psi \cong (\Delta', \Gamma')$.  By
    \cite[Corollary 2.6]{Hibi}, $\Gamma'$ is Cohen-Macaulay.  By assumption,
    $\Psi = \Delta' \setminus \Gamma'$ is Cohen-Macaulay, and since $\dim \Psi
    = \dim \Delta' = \dim \Gamma'$, it follows from~\cite[Prop
    1.2.9]{Bruns-Herzog} that $\Delta'$ is also Cohen--Macaulay.
\end{proof}

In the construction given in the course of the proof, the complexes $\Delta$
and $\Gamma$ occur as induced subcomplexes. If we are to abandon this
requirement, then our computations suggest that it suffices to add a single new
vertex. Based on this evidence, we offer the following conjecture.
\begin{conj}\label{conj1}
    Every Cohen--Macaulay relative complex $\Psi$ on ground set $[n]$ is a
    fully Cohen--Macaulay relative complex on ground set $[n+1]$. That is,
    for every  $(d-1)$-dimensional Cohen--Macaulay
    relative complex $\Psi = (\Delta,\Gamma)$ on ground set $[n]$, there are
    Cohen--Macaulay simplicial complexes $\Gamma' \subseteq
    \Delta' \subseteq 2^{[n+1]}$ of dimension $d-1$, such that
	$\Delta \setminus \Gamma = \Delta' \setminus \Gamma'$.
\end{conj}
We also offer a more precise conjecture on how the complexes $\Gamma'
\subset \Delta'$ can be obtained.

\begin{conj}\label{conj2}
    Let $\emptyset \neq \Gamma \subsetneq \Delta \subset 2^{[n]}$ be two
    simplicial complexes, such that the relative complex $(\Delta,\Gamma)$ is
    Cohen--Macaulay of dimension $d-1$ over some field $\kk$.  If $\Delta$ and
    $\Gamma$ have no common minimal non-faces, then the depth of $\kk[\Gamma]$
    is at least $d-1$.
\end{conj}

To see that Conjecture~\ref{conj2} implies Conjecture~\ref{conj1}, let $\Psi =
(\Delta,\Gamma)$ be a given presentation.  We can assume that $\Delta$ and
$\Gamma$ have no minimal non-faces in common.  Conjecture~\ref{conj2} then
assures us that $\kk[\Gamma]$ has depth $d-1$ and Theorem~\ref{thm:nice}
yields Conjecture~\ref{conj1}.

Instead of fixing the ground set, we may
instead consider the dimensions of the minimal faces in $\Psi = (\Delta,
\Gamma)$.
For a sequence $\alpha = (\alpha_1, \alpha_2, \alpha_3,\dotsc)$ of numbers and
$i \ge 0$ we set 
\[ 
    E^i\alpha \ := \ (\underbrace{0,\dotsc,0}_{i},\alpha_1, \alpha_2,
    \alpha_3,\dotsc) \, .
\]

\begin{thm}\label{thm:bjorner}
	For a vector $h = (h_0, \dotsc, h_{d}) \in \Znn^{d+1}$ and 
    numbers $a_1, \dotsc, a_r \in \Znn$, the following are
    equivalent:
	\begin{enumerate}[\rm (i)]
		\item $h = h(\Delta, \Gamma)$ for a shellable relative complex $(\Delta, \Gamma)$, whose minimal faces have cardinalities $a_1, \dotsc, a_r$;
		\item $h = h(\Delta, \Gamma)$ for a Cohen-Macaulay relative complex $(\Delta, \Gamma)$, whose minimal faces have cardinalities $a_1, \dotsc, a_r$;
		\item $h$ is the $h$-vector of a graded Cohen-Macaulay module (over some polynomial ring), whose generators have the degrees $a_1, \dotsc, a_r$.
        \item There exist M-sequences $\nu_1, \dotsc, \nu_r$ such that 
        \[
            h \ = \ 
            E^{a_1} \nu_1 + 
            E^{a_2} \nu_2 + 
            \cdots + 
            E^{a_r} \nu_r \, .
        \]
	\end{enumerate}
\end{thm}
The implications (i) $\Rightarrow$ (ii) $\Rightarrow$ (iii) are clear, and
(iii) $\Rightarrow$ (iv) is Proposition 5.2 of~\cite{stanley87}. 
In \emph{loc.~cit.} Anders Bj\"orner asked if the implication (iv)
$\Rightarrow$ (iii) also holds.

\begin{proof}
    We only need to show (iv) $\Rightarrow$ (i).  For each $i$, we can find a
    shellable simplicial complex $\Delta_i$ whose $h$-vector is $\nu_i$.
    Further, let $v_{i1},\dots,v_{i a_i}$ be new vertices and let $\Psi_i$ be
    the relative complex with faces $\{ F \cup \{v_{i1},\dots,v_{i a_i}\}
    \colon F \in \Delta_i\}$.  It is clear that any shelling order on
    $\Delta_i$ yields a shelling on $\Psi_i$, and that $h(\Psi_i) = E^{a_i}
    \nu_i$.  Finally, by taking cones if necessary, we may assume that all the
    $\Psi_i$ have the same dimension.  Then the disjoint union of the $\Psi_i$
    is the desired shellable relative complex.
\end{proof}

\bibliographystyle{amsalpha}
\bibliography{RelativeKKandM}

\end{document}